\documentclass[reqno]{amsart}
\usepackage{amssymb, cite}
\usepackage[colorlinks,plainpages,citecolor=magenta, linkcolor=blue , backref]{hyperref}
\usepackage[cp1251]{inputenc}
\usepackage[ukrainian]{babel}
\newtheorem{thm}{Theorem}

\newtheorem{prop}{Proposition}
\theoremstyle{definition}

\theoremstyle{remark}
\newtheorem{rem}{Remark}

\newcommand{\Real}{\mathbb R}
\newcommand{\eps}{\varepsilon}

\newcommand{\A}{\mathcal{A}}
\newcommand{\cH}{\mathcal{H}}
\newcommand{\cI}{\mathcal{I}}
\newcommand{\cJ}{\mathcal{J}}
\newcommand{\Ia}{\mathcal{I}_a}
\newcommand{\Ib}{\mathcal{I}_b}
\newcommand{\Iae}{\mathcal{I}_a^\varepsilon}
\newcommand{\Ibe}{\mathcal{I}_b^\varepsilon}
\newcommand{\Ae}{\mathcal{A}_\varepsilon}
\newcommand{\dom}{D}
\newcommand{\cL}{\mathcal{L}}
\newcommand{\rR}{\mathrm{R}_\zeta}
\newcommand{\rRz}{\mathrm{R}_z}

\newcommand{\Aao}{\mathring{A}_a}
\newcommand{\Abo}{\mathring{A}_b}
\newcommand{\Bo}{\mathring{B}}
\newcommand{\spn}{\mathop{\rm span}}
\begin{document}

\title[On spectrum of strings]{On spectrum of strings with $\delta'$-like perturbations of mass density}%
\author{Yuriy Golovaty}%



\begin{abstract}
We study the asymptotic behaviour of eigenvalues and eigenfunctions of a boundary value problem for  the Sturm-Liouville operator with ge\-ne\-ral boun\-da\-ry conditions and  the weight function perturbed by the so-called $\delta'$-like sequence $\eps^{-2}h(x/\eps)$.
The eigenvalue problem is realized as a family of non-self-adjoint matrix operators acting on the same Hilbert space and the norm resolvent convergence of this family is established.  We also prove the Hausdorff convergence of the perturbed spectra.
\end{abstract}
\maketitle
\section{Introduction}

The vibrating systems with added masses
have become the subject of research for mathematicians and physicists
since the time of Poisson and Bessel  \cite[Ch.2]{Sarpkaya2010}, and  an enormous number of studies have been devoted to these problems.
Many authors have investigated  properties of one-dimensional continua (strings and rods) with the mass density perturbed by the finite or infinite
sum $\sum_{k} M_k\delta(x-x_k)$, where $\delta$ is the Dirac function (see for instance \cite{Krylov1932, GantmakherKrein1950, TikhonovSamarskii1972, Timoshenko1972} and the references given there).
The mathematical models involving the $\delta$-functions are in general non suitable for 2D and 3D elastic systems, because the formal partial differential expressions which appear in the mo\-dels often have no mathematical meaning. Such models are also not adequate in the one-dimensional case, when
the added masses $M_k$ are large enough. The large adjoint mass can lead to a strong local reaction which brings about a considerable change in the basic form of the oscillations. But this reaction can not be described on the discrete set which is a support of singular distributions.
It is natural that the geometry of a small part of the vibrating system where the large  mass is loaded should also have an effect on eigenfrequencies and eigenvibrations.
Since works of E. S{\'a}nchez-Palencia \cite{PalenciaBook, MyFirstPaperOfSP, PalenciaHubertBook},  more adequate and more comp\-li\-cated mathematical models of  media with the concentrated masses have gained popularity; the asymptotic analysis began to be applied to  the spectral problems with the perturbed mass density having the form
\begin{equation*}
\rho_\eps(x)=\rho_0(x)+\sum_{k}\eps^{-m_k}h_k\left(\frac{x-x_k}{\eps}\right),
\end{equation*}
where $h_k$ are functions of compact support and $m_k\in\Real$.
The most interesting cases of the limit behaviour of eigenvalues and eigenfunctions as $\eps\to 0$ arise when the powers $m_k$ are greater than or equal to the dimension of vibrating system.

These improved  models have attracted considerable
attention in  the mathema\-ti\-cal literature over three past decades (see review \cite{LoboPerez2003}). The classic elastic systems such as strings, rods, membranes, plates and  bodies with the perturbed density $\rho_\eps(x)=\rho_0(x)+\eps^{-m}h(x/\eps)$ have been considered in \cite{Oleinik1987, OleinikM1988, OleinikK1988, GolovatyNazarovOleinikSoboleva1988, GolovatyjNazarovOleinik1990, OleinikIosifyanShamaevBook1990, GolovatyTrudy, Hrabchak1996},
where the convergence of spectra for each real $m$ and the complete asymptotic expansions of eigenvalues and eigenfunctions for  selected values of $m$ have been obtained. The influence of the concentrated masses on the spectral characteristics and oscillations of
junctions, the objects with very complicated geometry, has been studied
in  \cite{MelnikNazarov2001, Melnik2001, ChechkinMelnyk2012}.
The asymptotic behaviour of eigenvalues and eigenfunctions of membranes and bodies with many concentrated masses near the boundary has been investigated in
\cite{LoboPerez1993, LoboPerezSA1995, LoboPerez1995, Chechkin2005,  NazarovPerez2009}. In \cite{GolovatyLavrenyuk2000, GolovatyGomezLoboPerez2004} the asymptotic analysis has been applied to the spectral problems for membranes and plates with the  density perturbed in a thin neighbourhood of a closed smooth curve. The spectral problems on metric graphs that describe the eigenvibrations of elastic networks with heavy nodes have been studied in \cite{GolovatyHrabchak2007, GolovatyHrabchak2010}.

A characteristic feature of such problems is the presence of  perturbed density $\rho_\eps$ at the spectral parameter, which in turn leads to a self-adjoint operator realization  of the problem in a Hilbert space (a weighted Lebesgue space) that also depends on the small parameter. The study of families of operators acting on  varying  spaces entails some mathematical difficulties.    First of all, the question arises how to understand the convergence of such families. Next, if these operators do converge in some sense, does this convergence implies the convergence of their spectra (see  \cite[III.1]{OleinikIosifyanShamaevBook1990}, \cite{MelnikConvergence2001, MugnoloNittkaPost2010, Rosler2018} for more details).
Most of the above-mentioned publications deal with asymptotic approximations of eigenvalues and eigenfunctions; justifying  such asymptotics, the researchers used the theory of quasimodes \cite{LazutkinPDE5Viniti}, and therefore the question of the operator convergence can be avoided in the studies.

In this paper we consider  the Sturm-Liouville operators and investigate the eigenvalue problems with general boundary conditions and  the weight function perturbed by the so-called $\delta'$-like sequence $\eps^{-2}h(x/\eps)$.
By abandoning the self-adjointness, we realize the  perturbed problem as a family of non-self-adjoint matrix ope\-ra\-tors $\Ae$ acting on a fixed Hilbert space and prove the norm resolvent con\-ver\-gen\-ce of $\Ae$ as $\eps\to 0$. The operators $\Ae$ are certainly similar to self-adjoint ones for each $\eps$ and their spectra are real, discrete and simple. Surprisingly enough, the limit operator is essentially non-self-adjoint, because it possesses multiple eigenvalues with non-trivial Jordan cells. Actually the singularly perturbed problem gives us an example of some self-adjoint operators $T_\eps$ with compact resolvents acting on varying spaces $H_\eps$ that ``converge'' to a non-self-adjoint operator $T_0$ in space $H_0$. More precisely, the spectra of $T_\eps$ converge to the spectrum of $T_0$ in the Hausdorff sense, taking  account of the algebraic multiplicities of eigenvalues; moreover the limit position, as $\eps\to 0$,  of the  eigensubspaces of $T_\eps$ can be described by means of the root subspaces of $T_0$.

Note that a partial case of the problem, namely the Sturm-Liouville operator without a potential  subject to the Dirichlet type boundary condition, was previously studied in \cite{GolovatyNazarovOleinikSoboleva1988}. In Theorem~9,  the Hausdorff convergence of the perturbed spectrum  to some limit set was proved. This limit set was treated as a union of spectra of three self-adjoint operators (cf. Theorem~\ref{ThmSpA} below), but the limit operator was not  constructed and the question of eigenvalue multiplicity was not  discussed.

We use the following notation. Let $L_2(r, I)$ be the weighted Lebesgue space with the norm
\begin{equation*}
  \|f\|_{L_2(r, I)}=\left(\int_Ir(x)|f(x)|^2\,dx\right)^{1/2},
\end{equation*}
provided $r$ is positive.
Throughout the paper, $W_2^k(I)$ stands for the Sobolev space of functions defined on $I\subset\Real$ that belong to $L_2(I)$
together with their derivatives up to the order $k$. The norm in $W_2^k(I)$ is given by
$$
    \|f\|_{W_2^k(I)}
        := \bigl(\|f^{(k)}\|^2_{L_2(I)} + \|f\|^2_{L_2(I)} \bigr)^{1/2},
$$
where $\|f\|_{L_2(I)}$ is the usual $L_2$-norm.
The spectrum, point spectrum and resolvent set of a linear operator $T$ are
denoted by $\sigma(T)$, $\sigma_p(T)$ and $\rho(T)$, respectively, and the Hilbert space adjoint operator of $T$ is $T^*$. For any complex number $z\in \rho(T)$, the resolvent operator $\rRz(T)$ is defined by $\rRz(T)=(T-z)^{-1}$. Also, we will sometimes  abuse notation and write column vectors as row vectors.

\section{Statement of  Problem}
Let $\cI=(a,b)$ be a finite interval in $\Real$ containing the origin and $\eps$ be a small positive parameter.
Set $\Ia=(a,0)$, $\Ib=(0, b)$, $\Iae=(a,-\eps)$, $\Ibe=(\eps,b)$ and $\cJ=(-1,1)$.
We study the limiting behavior as $\eps\to 0$ of eigenvalues
$\lambda^\eps$ and eigenfunctions $y_\eps$ of  the  problem
\begin{align}\label{PertPrbEq}
        -&y_\eps''+q(x)y_\eps=\lambda^\eps r_\eps(x) y_\eps,\quad x\in \cI,
        \\ \label{PertPrbCondA}
        &y_\eps(a)\cos\alpha+y_\eps'(a)\sin\alpha=0,
        \\\label{PertPrbCondB}
         &y_\eps(b)\cos\beta+y_\eps'(b)\sin\beta=0
\end{align}
with the singularly perturbed weight function
\begin{equation*}
    r_\eps(x)=
    \begin{cases}
        r(x), & x\in \Iae\cup \Ibe,\\
        \eps^{-2}h(\eps^{-1} x), & x\in (-\eps, \eps).
    \end{cases}
\end{equation*}
Assume that $\alpha, \beta \in \Real$,  $q, r \in L^\infty(\cI)$ and  $h \in L^\infty(\cJ)$;   $r$ and $h$ are uniformly positive.

For any fixed real $\alpha$, $\beta$ and positive $\eps$ small enough, problem \eqref{PertPrbEq}--\eqref{PertPrbCondB} admits a self-adjoint realization in the weighted  space $L_2(r_\eps, \cI)$.
Let us consider the Sturm-Liouville differential expression $\tau(\phi)=-\phi''+q\phi$. We introduce the operator
$T_\eps$ defined by $T_\eps\phi=r_\eps^{-1}\tau(\phi)$ on functions $\phi\in W_2^2(\cI)$ obeying boundary conditions \eqref{PertPrbCondA} and \eqref{PertPrbCondB}.
Hence $\{T_\eps\}_{\eps>0}$ is a family of self-adjoint operators in the varying Hilbert spaces $L_2(r_\eps, \cI)$. Of course the spectrum of $T_\eps$ is real, discrete and simple.

Problem \eqref{PertPrbEq}--\eqref{PertPrbCondB} can be also associated with a non-self-adjoint matrix operator in the fixed Hilbert space
$\cL=L_2(r,\Ia)\times L_2( h,\cJ)\times L_2(r,\Ib)$ as follows. Subsequently, we will write boundary conditions \eqref{PertPrbCondA} and \eqref{PertPrbCondB} for a function $\phi$ as $\ell_a \phi=0$ and $\ell_b \phi=0$ respectively.
Let us introduce  the new variable $t=x/\eps$ and set $w_\eps(t)=y_\eps(\eps t)$. Then the eigenvalue problem can be written in the form
\begin{align}\label{PertPrEq(a,0)}
&-y_\eps''+q(x)y_\eps=\lambda^\eps r(x) y_\eps,\quad x\in \Iae,
   \qquad \ell_a y_\eps=0,\\\label{PertPrEq(-1,1)}
&-w_\eps''+\eps^2 q(\eps t) w_\eps=\lambda^\eps  h(t) w_\eps,\quad t\in \cJ,\\
        \label{PertPrEq(0,b)}
&-y_\eps''+q(x)y_\eps=\lambda^\eps r(x) y_\eps,\quad x\in \Ibe, \qquad
\ell_b y_\eps=0
\end{align}
with the  coupling conditions
\begin{gather}\label{PertPrTransConds0}
   y_\eps(-\eps)=w_\eps(-1),\quad y_\eps(\eps)=w_\eps(1),  \\\label{PertPrTransConds1}
   \eps y_\eps'(-\eps)= w'_\eps(-1),\quad \eps y_\eps'(\eps)= w'_\eps(1).
\end{gather}
Let $\Aao$ be the operator in $L_2(r,\Ia)$ that is defined by $\Aao\phi=r^{-1}\tau(\phi)$ on functions $\phi$ belonging to the set
$\dom(\Aao)=\left\{\phi\in W_2^2(\Ia)\colon \ell_a\phi=0\right\}$.
Similarly,  let $\Abo$ be the ope\-rator in $L_2(r,\Ib)$ such that $\Abo\phi=r^{-1}\tau(\phi)$ and $\dom(\Abo)=\left\{\phi\in W_2^2(\Ib)\colon \ell_b\phi=0\right\}$.
We also introduce the operator  $\Bo=- h^{-1}\frac{d^2}{dt^2} $ in
$L_2( h, \cJ)$ with domain $\dom(\Bo)= W_2^2(\cJ)$
and its  potential perturbation $\Bo_\eps=\Bo+\eps^2\frac{q(\eps t)}{ h(t)}$.

Let us consider the matrix operator
\begin{equation*}
    \Ae=
    \begin{pmatrix}
        \mathring{A_a} & 0 & 0\\
        0 & \mathring{B}_\eps & 0\\
        0 & 0 & \mathring{A_b}
    \end{pmatrix}
\end{equation*}
 in $\cL$, acting on  the domain
\begin{multline*}
    \dom(\Ae)=\big\{(\phi_a, \psi, \phi_b)\in \dom(\mathring{A_a})\times \dom(\mathring{B}_\eps)\times \dom(\mathring{A_b})\colon\\
               \phi_a(-\eps)=\psi(-1),\:\;\phi_b(\eps)=\psi(1),\:\;
               \eps \phi_a'(-\eps)=\psi'(-1),\:\; \eps
               \phi_b'(\eps)= \psi'(1)\big\}.
\end{multline*}
A straightforward calculation shows that $\Ae$ is non-self-adjoint.
Note that the spectral equation $(\Ae-\lambda^\eps)Y_\eps=0$ is slightly different from eigenvalue problem \eqref{PertPrEq(a,0)}--\eqref{PertPrTransConds1}.
In fact, if we display the components of vector $Y_\eps$ by writing  $Y_\eps=(y^a_\eps, w_\eps, y^b_\eps)$, then we see at once that $y^a_\eps$ is a solution of \eqref{PertPrEq(a,0)} on the whole interval $\Ia$ (not only in $\Iae$), and $y^b_\eps$ is a solution of \eqref{PertPrEq(0,b)} on the whole interval $\Ib$. However, this ``extra information'', namely the extensions of the solutions to the intervals $\Ia$ and $\Ib$, does not prevent the operator $\Ae$ from adequately describing the spectrum and the eigenfunctions of \eqref{PertPrEq(a,0)}--\eqref{PertPrTransConds1} (or also \eqref{PertPrbEq}--\eqref{PertPrbCondB}), because of the uniqueness of such extensions.

\begin{prop}
$\sigma(\Ae)=\sigma(T_\eps)$.
\end{prop}
\begin{proof}
Fix a positive $\eps$. We will show that $\rho(\Ae)=\rho(T_\eps)$.
Suppose first that $\zeta\in \rho(T_\eps)$ and consider the equation $(\Ae-\zeta)Y=F$, where  $F$ belongs to $\cL$.
Suppose that $F=(f_a,f_0,f_b)$. Then we can construct the function
\begin{equation*}
    f(x)=
    \begin{cases}
       \; f_a(x) &\text{for } x\in \Iae,\\
       \; f_0(x/\eps) &\text{for } x\in (-\eps, \eps),\\
       \; f_b(x) &\text{for } x\in \Ibe
    \end{cases}
\end{equation*}
belonging to $L_2(r_\eps, \cI)$. Next, $y=(T_\eps-\zeta)^{-1}f$ is a unique solution  of the problem
\begin{align}\label{EqnOnIae}
&-y''+qy-\zeta r y=rf_a\quad \text{in } \Iae,\qquad \ell_ay=0,\\
&-\eps^2 y''+\eps^2q y-\zeta  h y= h f_0\quad \text{in }(-\eps, \eps),\\\label{EqnOnIbe}
&-y''+qy-\zeta r y=rf_b\quad \text{in } \Ibe, \qquad\ell_by=0,
\\\label{CouplConds}
&\quad [y]_{-\eps}=0,\quad[y]_{ \eps}=0,\quad  [y']_{-\eps}=0,\quad  [y']_{\eps}=0,
\end{align}
where $[y]_{x_0}$ is a jump of  $y$ at the point $x_0$.
Denote by $y_a$ the extension of $y$ from $\Iae$ to $\Ia$ as a solution of \eqref{EqnOnIae}. Recall that the right hand side $f_a$ is defined on the whole interval $\Ia$. This extension is uniquely defined.
Similarly, we denote by $y_b$ the solution of  \eqref{EqnOnIbe} in $\Ib$ such that $y_b(x)=y(x)$ for $x\in \Ibe$.
Then  vector
$Y(x)=(y_a(x), y(x/\eps),y_b(x))$ belongs to $\dom(\Ae)$ and solves $(\Ae-\zeta)Y=F$. The last equation admits a unique solution $Y$;
if we assume that there are more such solutions, then we immediately obtain a contradiction with the uniqueness of $y$. Therefore, $\rho(\Ae)\subset\rho(T_\eps)$.

Conversely, suppose $\zeta\in \rho(\Ae)$. We prove that $(T_\eps-\zeta)y=f$ is uniquely solvable for all $f\in L_2(r_\eps, \cI)$.
Given $f$, construct  the vector $F=(f_a(x),f(\eps t),f_b(x))$,
where $f_a$ and $f_b$ are the restrictions of $f$ to $\Ia$ and $\Ib$ respectively. Then the problem
\begin{align*}
&-\phi_a''+q\phi_a-\zeta r \phi_a=rf_a\quad \text{in } \Ia,\qquad \ell_a\phi_a=0,\\
&-\psi''+\eps^2q(\eps \,\cdot) \psi  -\zeta  h \psi= h f(\eps \,\cdot)\quad \text{in }\cJ,\\
&-\phi_b''+q\phi_b-\zeta r \phi_b=rf_b\quad \text{in } \Ib,\qquad \ell_b\phi_b=0,
\\
&\quad\phi_a(-\eps)=\psi(-1),\:\;\phi_b(\eps)=\psi(1),\:\;
               \eps \phi'_a(-\eps)=\psi'(-1),\:\; \eps
               \phi'_b(\eps)= \psi'(1).
\end{align*}
admits a unique solution $Y=(\Ae-\zeta )^{-1}F$. If
$Y=(\phi_a, \psi, \phi_b)$, then  function
\begin{equation*}
    y(x)=
    \begin{cases}
       \; \phi_a(x) &\text{for } x\in \Iae,\\
       \; \psi(x/\eps) &\text{for } x\in (-\eps, \eps),\\
       \; \phi_b(x) &\text{for } x\in \Ibe
    \end{cases}
\end{equation*}
is a solution of \eqref{EqnOnIae}--\eqref{CouplConds}. Since the spectrum of $T_\eps$ is discrete, the solvability of $(T_\eps-\zeta)y=f$ for all $f\in L_2(r_\eps, \cI)$ ensures $\zeta \in \rho(T_\eps)$, and hence
$\rho(T_\eps)\subset\rho(\Ae)$.
\end{proof}

\section{Norm Resolvent Convergence of $\Ae$}
In this section we will prove that the family of operators $\Ae$ converges in the norm resolvent sense as $\eps\to 0$.
Let $B$ be the restriction of $\Bo$ to the domain
\begin{equation*}
  \dom(B)=\left\{\psi\in \dom(\Bo)\colon \psi'(-1)=0,\;\: \psi'(1)=0\right\}.
\end{equation*}
We introduce the matrix  operator
\begin{equation*}
    \A=
    \begin{pmatrix}
        \mathring{A_a} & 0 & 0\\
        0 & B & 0\\
        0 & 0 & \mathring{A_b}
    \end{pmatrix}
\end{equation*}
in space $\cL$  acting on
\begin{equation*}
    \dom(\A)=\left\{(\phi_a, \psi, \phi_b)\in \dom(\Aao)\times \dom(B)\times \dom(\Abo)\colon               \phi_a(0)=\psi(-1),\:\;\phi_b(0)=\psi(1)\right\}.
\end{equation*}
This operator is associated with the eigenvalue problem
\begin{align}\label{LimitPrbYa}
        &-u''+qu=\lambda ru\quad\text{in }\in \Ia,
         \quad \ell_au=0,
         \\\label{LimitPrbW}
        &-w''=\lambda  hw,\quad \text{in }\in \cJ,
         \quad w'(-1)=0,\quad w'(1)=0,
         \\\label{LimitPrbYb}
        &-v''+qv=\lambda rv\quad\text{in }\in \Ib,
\quad \ell_bv=0,\\\label{LimitPrbY=W}
       &\quad u(0)=w(-1),\quad v(0)=w(1)
\end{align}
which can be regarded as \emph{the limit problem.}
The following assertion is one of the main results of this paper.

\begin{thm}\label{ThmResolventConvergence}
    The family of operators $\Ae$ converges to $\A$ as $\eps\to 0$ in the norm resolvent sense. In addition,
    \begin{equation}\label{ResolventEst}
      \|\rR(\Ae)-\rR(\A)\|\leq c\sqrt{\eps},
    \end{equation}
    the constant $c$ being independent of $\eps$.
\end{thm}

For the convenience of the reader we collect together the definitions
of all ope\-ra\-tors which will be used in the proof.
\begin{itemize}
\item
\textit{Operators $T_a^\eps(\zeta)$, $T_b^\eps(\zeta)$, $T_a(\zeta)$ and $T_b(\zeta)$.}
We endow  $\dom(\Bo)$ with the graph norm, i.e., the norm of the Sobolev space $W_2^2(\cJ)$.
Let   $T_a^\eps(\zeta)\colon \dom(\Bo)\to L_2(r,\Ia)$ be defined as follows.
Given $\zeta\in \mathbb{C}\setminus\Real$ and $\psi\in \dom(\Bo)$, we compute  $\psi(-1)$, find then a unique solution $u_a$ of the problem
\begin{equation}\label{PrbIaON}
 -u''+qu-\zeta r u=0\quad \text{in } \Ia,\qquad
 \ell_au=0,\quad u(-\eps)=\psi(-1)
\end{equation}
and finally set $T_a^\eps(\zeta)\psi=u_a$. Similarly, we define  $T_b^\eps(\zeta)\colon \dom(\Bo)\to L_2(r,\Ib)$ which solves the problem
\begin{equation}\label{PrbIbON}
 -v''+qv-\zeta r v=0\quad \text{in } \Ib,\qquad
  v(\eps)=\psi(1),\quad \ell_bv=0
\end{equation}
for given $\psi\in \dom(\Bo)$. Next, the operators $T_a(\zeta)$ and $T_b(\zeta)$
stand for $T_a^\eps(\zeta)$ and $T_b^\eps(\zeta)$, provided $\eps=0$. So $T_a(\zeta)$ (resp. $T_b(\zeta)$)  solves  problem \eqref{PrbIaON} (resp. \eqref{PrbIbON}) for given $\psi\in \dom(\Bo)$ and $\eps=0$.

\item
\textit{Operators $S_a^\eps(\zeta)$ and $S_b^\eps(\zeta)$.}
Suppose that $\dom(\Aao)$ and $\dom(\Abo)$ are equipped by the graph norms. These norms are equivalent to  the norms of $W_2^2(\Ia)$ and $W_2^2(\Ib)$ respectively. The operator $S_a^\eps(\zeta)\colon \dom(\Aao)\to L_2( h,\cJ)$ is defined by $S_a^\eps(\zeta)\phi=w_a$, where $w_a$ is a unique solution of
\begin{equation}\label{ProblemForSmuA}
  -w''+\eps^2 q(\eps\,\cdot) w=\zeta  h w\quad \text{in }  \cJ, \quad w'(-1)=\phi'(-\eps), \;\; w'(1)=0
\end{equation}
for given $\phi\in \dom(\Abo)$  and $\zeta\in \mathbb{C}\setminus\Real$. Similarly, operator $S_b^\eps(\zeta)\colon \dom(\Abo)\to L_2( h,\cJ)$ solves
\begin{equation}\label{ProblemForSmuB}
     -w''+\eps^2 q(\eps\,\cdot) w=\zeta  h w\quad \text{in }  \cJ, \qquad w'(-1)=0, \;\; w'(1)=\phi'(\eps)
  \end{equation}
for some $\phi\in \dom(\Aao)$ and $\zeta\in \mathbb{C}\setminus\Real$.

\item
\textit{Operator $B_\eps$.}
This operator is the restriction of $\Bo_\eps$ to the domain
\begin{equation*}
\dom(B_\eps)=\left\{\psi\in\dom(\Bo_\eps)\colon \psi'(-1)=0, \;\; \psi'(1)=0\right\}.
\end{equation*}

\item
\textit{Operators $A^\eps_a$, $A^\eps_b$, $A_a$ and $A_b$.}
Let  $A^\eps_a$ and  $A^\eps_b$ be the restrictions of $\Aao$ and $\Abo$ respectively to the domains $\dom(A^\eps_a)=\{\phi\in\dom(\Aao)\colon \phi(-\eps)=0\}$, $\dom(A^\eps_b)=\{\phi\in\dom(\Abo)\colon \phi(\eps)=0\}$.
The operators $A_a$ and $A_b$ stand for $A_a^\eps$ and $A_b^\eps$, provided $\eps=0$.
\end{itemize}

We now construct the resolvents of $\Ae$ and  $\A$ in the explicit form as follows. Fix $\zeta \in \mathbb{C}\setminus \Real$. First of all, note that
operators $T_a^\eps(\zeta)$, $T_b^\eps(\zeta)$, $S_a^\eps(\zeta)$ and $S_b^\eps(\zeta)$ are well-defined for such values of $\zeta$.
Moreover these operators are compact.
Given
$F=(f_a, f_0, f_b)\in \cL$, solve the equation $(\Ae-\zeta)Y=F$.
The first component  of   $Y=(\phi_a, \psi, \phi_b)$ is a solution of the Dirichlet type problem
\begin{equation*}
   -\phi''+q\phi-\zeta r \phi= r f_a\quad \text{in } \Ia,\qquad
   \ell_a\phi=0,\quad \phi(-\eps)=\psi(-1).
\end{equation*}
This solution can be represented as the sum of a solution of the non-homogeneous equation  subject to the homogeneous boundary  conditions and a solution of  \eqref{PrbIaON}:
\begin{equation}\label{YaRepres}
     \phi_a=\rR(A^\eps_a)f_a+T_a^\eps(\zeta)\psi.
\end{equation}
The same argument yields
\begin{equation}\label{YbRepres}
     \phi_b=\rR(A^\eps_b)f_b+T_b^\eps(\zeta)\psi.
\end{equation}
The middle element $\psi$ of $Y$ is a solution of the Neumann type problem
\begin{equation*}
     -\psi''+\eps^2 q(\eps \,\cdot) \psi-\zeta  h \psi= h f_0\quad \text{in }  \cJ, \quad \psi'(-1)=\eps \phi_a'(-\eps), \quad \psi'(1)=\eps \phi_b'(\eps),
  \end{equation*}
and it can be written as
\begin{equation}\label{YRepres}
    \psi=\rR(B_\eps)f_0+\eps S_a^\eps(\zeta)\phi_a+\eps S_b^\eps(\zeta)\phi_b.
\end{equation}
Then \eqref{YaRepres}--\eqref{YRepres} taken together yield
\begin{align*}
        \phi_a-T_a^\eps(\zeta)\psi=\rR(A^\eps_a)f_a,\\
      -\eps S_a^\eps(\zeta)\phi_a +\psi-\eps S_b^\eps(\zeta)\phi_b=\rR(B_\eps)f_0,\\
        -T_b^\eps(\zeta)\psi+\phi_b=\rR(A^\eps_b)f_b.
\end{align*}
It follows that the resolvent of $\Ae$ has the form
\begin{equation}\label{ResolventOfAe}
\rR(\Ae)=\cH_\eps(\zeta)^{-1} \mathcal{R}_\eps(\zeta),
\end{equation}
where
\begin{gather}\label{OprCReps}
\mathcal{R}_\eps(\zeta)=\begin{pmatrix}
              \rR(A^\eps_a)         & 0 & 0\\
    0 &         \rR(B_\eps)     & 0\\
         0              & 0 &   \rR(A^\eps_b)
\end{pmatrix},
\\\label{OprCHeps}
\cH_\eps(\zeta)=\begin{pmatrix}
         E              & -T_a^\eps(\zeta) & 0\\
    -\eps S_a^\eps(\zeta) &        E       & -\eps S_b^\eps(\zeta)\\
         0              & -T_b^\eps(\zeta) & E
\end{pmatrix},
\end{gather}
and $E$ denotes the identity operator in the corresponding spaces.
We shall prove below that  $\cH_\eps(\zeta)$ is invertible for $\eps$ small enough.

Now we consider  the equation $(\A-\zeta)Y=F$ for $F\in \cL$. In the coordinate representation we have $(\Aao-\zeta)\phi_a=f_a$, $(B-\zeta)\psi=f_0$ and $(\Abo-\zeta)\phi_b=f_b$,
where $Y=(\phi_a,\psi,\phi_b)$ and $F=(f_a,f_0,f_b)$.
Obviously,  $\psi=\rR(B)f_0$. The functions $\phi_a$ and $\phi_b$ are solutions of the problems
\begin{align*}
  &-\phi''+q\phi-\zeta r \phi= r f_a\quad \text{in } \Ia,\qquad
   \ell_a\phi=0,\quad \phi(0)=\psi(-1);\\
   &-\phi''+q\phi-\zeta r \phi= r f_b\quad \text{in } \Ib,\qquad
  \phi(0)=\psi(-1),\quad \ell_b\phi=0
\end{align*}
respectively. By reasoning similar to that for \eqref{YaRepres} and  \eqref{YbRepres}, we find
\begin{equation*}
    \phi_a=\rR(A_a)f_a+T_a(\zeta)\rR(B)f_0,\qquad
    \phi_b=\rR(A_b)f_b+T_b(\zeta)\rR(B)f_0.
\end{equation*}
Hence the resolvent of $\A$ can be written in the form
\begin{equation}\label{ResolventA}
    \rR(\A)=
    \begin{pmatrix}
        {\rm R}_\zeta(A_a) & T_a(\zeta)\rR(B) & 0\\
              0          & \phantom{T_a}\rR(B)        & 0\\
              0          & T_b(\zeta)\rR(B) & {\rm R}_\zeta(A_b)
    \end{pmatrix}.
\end{equation}

To compare the resolvents of $\Ae$ and $\A$, we need some auxiliary  assertions.

\begin{prop}\label{PropAmuEpsAB}
The operators  $A^\eps_a$, $A^\eps_b$ and $B_\eps$ converge as $\eps\to 0$
to $A_a$, $A_b$ and $B$ respectively  in the norm resolvent sense. Moreover
\begin{gather}\label{EstRAeabMinusRAab}
    \left\|\rR(A^\eps_a)-\rR(A_a)\right\|\leq C_1\sqrt{\eps}, \qquad
    \left\|\rR(A^\eps_b)-\rR(A_b)\right\|\leq C_2\sqrt{\eps},\\\label{EsrRBeMinusRB}
    \left\|\rR(B_\eps)-\rR(B)\right\|\leq C_3 \eps^2,
\end{gather}
where the constants $C_k$ do not depend on $\eps$.
\end{prop}
\begin{proof}
Fix $\zeta\in \mathbb{C}\setminus\mathbb{R}$ and let compare the elements $u_\eps=\rR(A^\eps_b)f$ and $u=\rR(A_b)f$ for given $f\in L_2(r,\Ib)$.
Since $u_\eps$ and $u$ solve the problems
\begin{align*}
    &-u_\eps''+qu_\eps-\zeta r u_\eps=rf\quad \text{in } \Ib,\qquad u_\eps(\eps)=0,\quad \ell_bu_\eps=0;\\
    &-u''+qu-\zeta r u=rf\quad \text{in } \Ib,\qquad\quad u(0)=0,\quad
    \ell_bu=0,
\end{align*}
they are related by  equality
\begin{equation*}
   u_\eps(x)=u(x)-\frac{u(\eps)}{z(\eps)}\,z(x),\qquad x\in \Ib,
\end{equation*}
where $z$ is a solution of the problem
\begin{equation}\label{PrbW}
  -z''+qz-\zeta r z=0\quad \text{in } \Ib,\qquad z(0)=1, \quad \ell_bz=0.
\end{equation}
Obviously, $z(\eps)$ is different from zero for $\eps$ small enough.  Then we have
\begin{equation*}
  \|u_\eps-u\|_{L_2(r,\Ib)}\leq \frac{|u(\eps)|}{|z(\eps)|}\,\|z\|_{L_2(r,\Ib)}
    \leq c_1 |u(\eps)|\leq c_2 \sqrt{\eps}\,\|u\|_{W_2^1(\Ib)},
\end{equation*}
because $z(\eps)\to 1$ as $\eps\to 0$ and
\begin{equation*}
    |u(\eps)|=\left|\int_0^\eps u'(x)\,dx\right|\leq c_3 \sqrt{\eps}\,\|u\|_{W_2^1(\Ib)}.
\end{equation*}
Observe that $\rR(A_b)$ is a bounded operator from $L_2(r,\Ib)$ to the domain of $A_b$ equipped with the graph norm. Since the domain is a subspace of $W_2^1(\Ib)$, there exists a constant
$c_4$ independent of $f$ such that
\begin{equation*}
    \|u\|_{W_2^1(\Ib)}\leq c_4\|f\|_{L_2(r,\Ib)}.
\end{equation*}
Therefore $\left\|\bigl(\rR(A^\eps_b)-\rR(A_b)\bigr)f\right\|_{L_2(r,\Ib)}\leq C_2\sqrt{\eps}\,\|f\|_{L_2(r,\Ib)}$, which establishes  the norm resolvent convergence $A^\eps_b\to A_b$ as $\eps\to 0$ and the corresponding estimate in \eqref{EstRAeabMinusRAab}. The proof for the operators $A^\eps_a$ is similar to that just given.

We now turn to the operators $B_\eps$ and first we establish that $\|\rR(B_\eps)\|\leq c$ for all $\eps$ small enough. Given $g\in L_2( h,\cJ)$, consider $w_\eps=\rR(B_\eps)g$ which solves
\begin{equation*}
     -w_\eps''+\eps^2 q(\eps \,\cdot) w_\eps-\zeta  h w_\eps= h g\quad \text{in }  \cJ, \quad w_\eps'(-1)=0, \quad w_\eps'(1)=0.
  \end{equation*}
Recall that  $q$ and $h$ are bounded in $\cI$ and $\cJ$ respectively,
and $h$ is uniformly positive on $\cI$. Then we have
\begin{multline*}
  \|\rR(B_\eps)g\|_{L_2( h,\cJ)}=\left\|\rR(B)\big(g-\eps^2 q(\eps \,\cdot)h^{-1} w_\eps\big)\right\|_{L_2( h,\cJ)}
  \leq\|\rR(B)g\|_{L_2( h,\cJ)}\\
  +\eps^2\|q\|_{L^\infty(\cI)}\,\|h^{-1}\|_{L^\infty(\cJ)}\,\|w_\eps\|_{L_2( h,\cJ)}\leq c_0\|g\|_{L_2( h,\cJ)}+c_1\eps^2\|\rR(B_\eps)g\|_{L_2( h,\cJ)}
\end{multline*}
and therefore
\begin{equation}\label{EstRBe}
  \|\rR(B_\eps)g\|_{L_2( h,\cJ)}\leq\frac{c_0}{1-c_1\eps^2}\,\|g\|_{L_2( h,\cJ)}\leq c\|g\|_{L_2( h,\cJ)}
\end{equation}
if $\eps$ is small enough.

Next, we set $w=\rR(B)g$. Then the difference $s_\eps=w_\eps-w$ solves the problem
\begin{equation*}
     -s_\eps''-\zeta  h s_\eps= -\eps^2 q(\eps \,\cdot) w_\eps\quad \text{in }  \cJ, \quad s_\eps'(-1)=0, \quad s_\eps'(1)=0.
  \end{equation*}
Hence in view of \eqref{EstRBe} we deduce
\begin{multline*}
  \|(\rR(B_\eps)-\rR(B))g\|_{L_2( h,\cJ)}= \|s_\eps\|_{L_2( h,\cJ)}\leq c_2\eps^2\|w_\eps\|_{L_2( h,\cJ)}\\=c_2\eps^2\|\rR(B_\eps)g\|_{L_2( h,\cJ)}\leq c_3\eps^2\|g\|_{L_2( h,\cJ)},
\end{multline*}
which finishes the proof.
\end{proof}

\begin{prop}\label{PropTaTbSaSb}
(i)
For each $\zeta\in \mathbb{C}\setminus\mathbb{R}$, we have the bounds
 \begin{equation*}
   \|T_a^\eps(\zeta)-T_a(\zeta)\|\leq c\eps, \qquad \|T_b^\eps(\zeta)-T_b(\zeta)\|\leq c\eps,
 \end{equation*}
the constant $c$ being independent of $\eps$.

(ii) There exists  constant $C$  such that
\begin{equation*}
 \|S_a^\eps(\zeta)\|+\|S_b^\eps(\zeta)\|\leq C
\end{equation*}
for all $\eps$ small enough.
\end{prop}

\begin{proof} \textit{(i)} Let us show that $T_b^\eps(\zeta)$ converge to $T_b(\zeta)$ in the norm  as $\eps\to 0$. The same proof remains valid for $T_a^\eps(\zeta)$.  Suppose that $u_\eps=T_b^\eps(\zeta)\psi$ is a solution of \eqref{PrbIbON} for given $\psi\in \dom(\Bo)$. It is easily seen that
\begin{equation*}
    u_\eps(x)=\frac{\psi(1)}{z(\eps)}\,z(x),\quad x\in\Ia,
\end{equation*}
where $z$ is defined by~\eqref{PrbW}. If $u=T_b(\zeta)\psi$, then we  have $u=\psi(1)z$.
Hence
   \begin{multline*}
    \|(T_b^\eps(\zeta)-T_b(\zeta))\psi\|_{L_2(r,\Ib)}
    =\left\|\frac{\psi(1)}{z(\eps)}\,z
    -\psi(1)z\right\|_{L_2(r,\Ib)}
    \\\leq \left|\frac{z(\eps)-1}{z(\eps)}\right|\,|\psi(1)|\,\|z\|_{L_2(r,\Ib)}
    \leq
    c_1\eps\,\|\psi\|_{\dom(\Bo)},
   \end{multline*}
because $z$ belongs to  $C^1(\Ib)$  and $z(0)=1$.
Recall also  that $\dom(\Bo)=W_2^2(\cJ)$ and hence
$\|\psi\|_{C(\cJ)}\leq C\|\psi\|_{\dom(\Bo)}$ by the Sobolev embedding theorem.

\textit{(ii)}
For each $\phi\in \dom(\Abo)$, the function $w_\eps=S_b^\eps(\zeta)\phi$ is a solution of \eqref{ProblemForSmuB} and satisfies the estimate
$\|w_\eps\|_{L_2( h,\cJ)}\leq c_2 \left|\phi'(\eps)\right|$
with a constant $c_2$  independent of $\eps$, since the resolvents $\rR(B_\eps)$ are uniformly bounded on $\eps$ by  Proposition~\ref{PropAmuEpsAB}. The trace operator $j_\eps\colon \dom(\Abo)\to \mathbb{C}$, $j_\eps \phi=\phi'(\eps)$, is also uniformly bounded on $\eps$. Therefore
\begin{equation*}
    \|S_b^\eps(\zeta)\phi\|_{L_2( h,\cJ)} =\|w_\eps\|_{L_2( h,\cJ)}\leq  C \|\phi\|_{W^2_2(\Ib)}.
\end{equation*}
The same proof works for $S_a^\eps(\zeta)$.
\end{proof}

We are now in a position to prove Theorem~\ref{ThmResolventConvergence}.
In view of Proposition~\ref{PropTaTbSaSb}, we conclude that the family of matrix operators $\cH_\eps(\zeta)$, given by \eqref{OprCHeps}, converges as $\eps\to 0$ towards
\begin{equation*}
    \cH(\zeta)=\begin{pmatrix}
         E              & -T_a(\zeta) & 0\\
    0 &        E       & 0\\
         0              & -T_b(\zeta) & E
\end{pmatrix}
\end{equation*}
in the norm. Moreover $\|\cH_\eps(\zeta)-\cH(\zeta)\|\leq c_1\eps$.
Observe  that  $\cH(\zeta)$ is invertible and
\begin{equation*}
    \cH(\zeta)^{-1}=\begin{pmatrix}
         E              & T_a(\zeta) & 0\\
    0 &        E       & 0\\
         0              & T_b(\zeta) & E
\end{pmatrix}.
\end{equation*}
Therefore $\cH_\eps(\zeta)$ is also invertible for $\eps$  small enough, and
\begin{equation}\label{EstHinverse}
  \|\cH_\eps(\zeta)^{-1}-\cH(\zeta)^{-1}\|\leq c_2 \eps.
\end{equation}
Recalling \eqref{ResolventOfAe} and applying  Proposition~\ref{PropAmuEpsAB}, we deduce
\begin{multline*}
\rR(\Ae)=\cH_\eps(\zeta)^{-1} \mathcal{R}_\eps(\zeta)\\
=
\begin{pmatrix}
         E              & -T_a^\eps(\zeta) & 0\\
    -\eps S_a^\eps(\zeta) &        E       & -\eps S_b^\eps(\zeta)\\
         0              & -T_b^\eps(\zeta) & E
\end{pmatrix}^{-1}
\begin{pmatrix}
              \rR(A^\eps_a)         & 0 & 0\\
    0 &         \rR(B_\eps)     & 0\\
         0              & 0 &   \rR(A^\eps_b)
\end{pmatrix}
\\\to
\begin{pmatrix}
         E  & T_a(\zeta) & 0\\
         0  &        E   & 0\\
         0  & T_b(\zeta) & E
\end{pmatrix}
\begin{pmatrix}
    \rR(A_a)&0&0\\
    0&\rR(B)& 0\\
    0&0&\rR(A_b)
\end{pmatrix}
\\
=
    \begin{pmatrix}
        {\rm R}_\zeta(A_a) & T_a(\zeta)\rR(B) & 0\\
              0          & \phantom{T_a}\rR(B)        & 0\\
              0          & T_b(\zeta)\rR(B) & {\rm R}_\zeta(A_b)
    \end{pmatrix}= \rR(\A)  \quad\text{as } \eps\to 0,
\end{multline*}
by \eqref{ResolventA}.
Estimate \eqref{ResolventEst} follows from equality
\begin{equation*}
  \rR(\Ae)-\rR(\A)=\cH_\eps(\zeta)^{-1}(\mathcal{R}_\eps(\zeta)
  -\mathcal{R}(\zeta))
  -(\cH_\eps(\zeta)^{-1}-\cH(\zeta)^{-1})\mathcal{R}(\zeta)
\end{equation*}
and bounds \eqref{EstRAeabMinusRAab}, \eqref{EsrRBeMinusRB} and \eqref{EstHinverse}. Here $\mathcal{R}(\zeta)={\rm diag}\,\big\{\rR(A_a), \rR(B), \rR(A_b)\big\}$.

\section{Spectrum of $\A$}

The limit operator
\begin{equation*}
    \A=
    \begin{pmatrix}
        \mathring{A_a} & 0 & 0\\
        0 & B & 0\\
        0 & 0 & \mathring{A_b}
    \end{pmatrix}, \quad
    \begin{aligned}
        \dom(\A)=\Big\{(\phi_a, \psi,\phi_b)&\,\in \dom(\Aao)\times \dom(B)\times \dom(\Abo)\colon\\
               &\phi_a(0)=\psi(-1),\:\;\phi_b(0)=\psi(1)\Big\}.
    \end{aligned}
\end{equation*}
 constructed above  is non-self-adjoint. Direct computations show that the adjoint operator $\A^*$ in $\cL$ has the form
\begin{equation*}
    \A^*=
    \begin{pmatrix}
        A_a & 0 & 0\\
        0 & \mathring{B} & 0\\
        0 & 0 & A_b
    \end{pmatrix}, \;\;
    \begin{aligned}
       \dom(\A^*)=\Big\{(\phi_a, \psi, \phi_b)&\,\in \dom(A_a)\times \dom(\Bo)\times \dom(A_b)\colon\\
               &\phi_a'(0)=\psi'(-1),\:\;\phi_b'(0)=\psi'(1)\Big\}.
    \end{aligned}
\end{equation*}
In what follows we will denote by $u_\lambda$, $v_\lambda$ and $w_\lambda$ the eigenfunctions of $A_a$, $A_b$ and $B$ respectively which correspond to an eigenvalue $\lambda$. So $u_\lambda$, $v_\lambda$ and $w_\lambda$ are non-trivial solutions of the problems
\begin{align}\label{EVPrAa}
  &-u''+qu=\lambda r u\quad \text{in } \Ia,\qquad \ell_au=0,\quad u(0)=0;
   \\\label{EVPrAb}
  &-v''+qv=\lambda r v\quad \text{in } \Ib,\qquad v(0)=0,\quad
    \ell_bv=0;\\\label{EVPrB}
  &-w''=\lambda hw  \quad \text{in } \cJ, \qquad w'(-1)=0, \quad w'(1)=0
\end{align}
respectively. Let us normalize these eigenfunctions
by setting
\begin{equation}\label{Normalization}
\|u_\lambda\|_{L_2(r,\Ia)}=\|v_\lambda\|_{L_2(r,\Ib)}=\|w_\lambda\|_{L_2( h,\cJ)}=1.
\end{equation}
Denote also by $X_\lambda$ the root subspace of $\A$ for $\lambda$, that is
\begin{equation*}
  X_\lambda=\spn\left\{\ker (\A-\lambda)^k\colon k\in \mathbb{N}\right\}.
\end{equation*}
The eigenvectors and root vectors of a non-self-adjoint operator are also called genera\-li\-zed eigenvectors. So $X_\lambda$ is a subspace of the generalized eigenfunctions corresponding to the eigenvalue $\lambda$.

\begin{thm}\label{ThmSpA}
(i) The spectrum of $\A$ is real and discrete, and
   \begin{equation}\label{sigma(A)}
    \sigma(\A)=\sigma(A_a)\cup\sigma(B)\cup\sigma(A_b).
   \end{equation}

(ii)  If $\lambda$ belongs  to only one of the sets $\sigma(A_a)$, $\sigma(B)$ or $\sigma(A_b)$, then $\lambda$ is a simple eigenvalue of $\A$.

(iii) If  $\lambda\in\sigma(A_a)\cap \sigma(A_b)$, but $\lambda$ is not an eigenvalue of $B$, then $\lambda$ is a double eigenvalue and $X_\lambda=\ker (\A-\lambda E)$.

(iv)  Suppose that $\lambda$ belongs to $\sigma(A_a)\cap \sigma(B)$ (resp. $\sigma(A_b)\cap\sigma(B)$), but $\lambda$ is not an eigenvalue of $A_b$ (resp. $A_a$), then $\lambda$ is a double eigenvalue of $\A$.
Finally, if $\lambda\in\sigma(A_a)\cap \sigma(A_b)\cap \sigma(B)$, then $\lambda$ is an  eigenvalue of $\A$ with multiplicity $3$.
 In both the cases we have  $X_\lambda=\ker(\A-\lambda)^2$, but $X_\lambda\neq\ker(\A-\lambda)$.
\end{thm}
\begin{proof}
\textit{(i)}
Equality \eqref{sigma(A)}  follows directly from the explicit representation \eqref{ResolventA} of $\rR(\A)$. Indeed,  each of spectra $\sigma(A_a)$, $\sigma(A_b)$ and $\sigma(B)$ is contained in the spectrum of $\A$. If $\zeta$ does not  belongs to set $\sigma(A_a)\cup\sigma(A_b)\cup\sigma(B)$, then not only $\rR(A_a)$, $\rR(A_b)$, $\rR(B)$, but also $T_a(\zeta)$ and $T_b(\zeta)$ are bounded, because
in this case problems \eqref{PrbIaON} and \eqref{PrbIbON} for $\eps=0$ are uniquely solvable for all $\psi\in W_2^2(\cJ)$. Therefore operator $\rR(\A)$ is also bounded.
Operators $A_a$, $A_b$ and $B$  associated with  eigenvalue problems \eqref{EVPrAa}, \eqref{EVPrAb} and \eqref{EVPrB}   are self-adjoint and have compact resolvents. Consequently $\sigma(A)$ is real and discrete.

\textit{(ii)}
Observe that the spectra of $A_a$,  $A_b$ and $B$ are simple.
A trivial verification shows that if $\lambda$ belongs  to only one of the sets $\sigma(A_a)$, $\sigma(A_b)$ or $\sigma(B)$, then $\lambda$ is a simple eigenvalue of $\A$ with  eigenvector $(u_\lambda,0,0)$ if $\lambda\in \sigma(A_a)$, and $(0,0, v_\lambda)$ if $\lambda\in \sigma(A_b)$, and
$(T_a(\lambda)w_\lambda,w_\lambda,T_b(\lambda)w_\lambda)$ if $\lambda\in \sigma(B)$.

\textit{(iii)}
In the case $\lambda\in\sigma(A_a)\cap \sigma(A_b)$ and  $\lambda\not\in \sigma(B)$, there are two linearly independent eigenvectors $U=(u_\lambda,0,0)$ and $V=(0,0, v_\lambda)$.
Moreover, equation $(\A-\lambda)Y=c_1U+c_2V$ is unsolvable for any $c_1$ and $c_2$ such that $c_1^2+c_2^2\neq 0$. If for instance  $c_1$ is different from zero, then problem
\begin{equation}\label{ProblemIII}
     -u''+qu-\lambda r u=c_1r u_\lambda \quad \text{in }\Ia, \qquad\ell_a u=0,\quad u(0)=0
\end{equation}
has no solutions. Suppose, contrary to our claim,  that such solution exists. Then multiplying   equation  \eqref{ProblemIII} by $u_\lambda$ and integrating by parts yield $c_1\|u_\lambda\|^2_{L_2(r,\Ia)}=0$.
Therefore $X_\lambda=\ker(\A-\lambda)$ and $\dim X_\lambda=2$.

\textit{(iv)}
Suppose that $\lambda\in \sigma(A_a)\cap \sigma(B)$ and $\lambda\not\in \sigma(A_b)$. In this case there exists the eigenvector  $U=(u_\lambda,0,0)$. Furthermore, we will show that the equation $(\A-\lambda)U_*=U$ is solvable.  We are thus looking for a solution $U_*=(u,w,v)$ of
\begin{align}\label{UstarProblem}
     &-u''+qu-\lambda r u=r u_\lambda \quad \text{in }\Ia, &&\ell_a u=0,\quad u(0)=w(-1);\\\label{WstarProblem}
&-w''-\lambda  h w=0\quad \text{in }\cJ, && w'(-1)=0,\quad w'(1)=0;\\\label{VstarProblem}
    &-v''+qv-\lambda r v=0 \quad \text{in } \Ib, && v(0)=w(1),\quad
    \ell_b v=0.
\end{align}
Obviously, $w=c_0w_\lambda$ for some constant $c_0$, where $w_\lambda$ is a normalized eigenfunction of $B$. Then  \eqref{VstarProblem} admits a unique solution $v_*=c_0\,T_b(\lambda)w_\lambda$ for each $c_0$, since $\lambda\in \varrho(A_b)$. Next,   \eqref{UstarProblem} is in general unsolvable, since $\lambda$ is a point of $\sigma(A_a)$. But  we have the free parameter $c_0$ in the boundary condition; \eqref{UstarProblem} with the condition $u(0)=c_0 w_\lambda(-1)$ is solvable if and only if
\begin{equation}\label{ConstantC}
    c_0=\frac1{w_\lambda(-1)u'_\lambda(0)}.
\end{equation}
This equality can be easily obtained by multiplying the equation in  \eqref{UstarProblem} by $u_\lambda$ and integrating by parts. Remark that
both of the values $w_\lambda(-1)$ and $u'_\lambda(0)$ are different from zero. If $u_0$ is a solution of \eqref{UstarProblem}, then operator $\A$ has a root vector
$$
U_*=\left(u_0,c_0\,w_\lambda, c_0\,T_b(\lambda)w_\lambda\right),
$$
where $c_0$ is given by \eqref{ConstantC}.
Hence, the subspace $X_\lambda$ is a linear span of  the eigenvector $U$ and the root vector $U_*$. In addition, there are no other root vectors, because the equation $(\A-\lambda)Y=U_*$ leads to the problem
\begin{equation}\label{UnsolvProblem}
    -w''-\lambda  h w=chw_\lambda\quad \text{in }\cJ,\qquad w'(-1)=0,\quad w'(1)=0,
\end{equation}
which is unsolvable for $c\neq 0$.
The case $\lambda\in\sigma(A_b)\cap \sigma(B)$ and $\lambda\not\in\sigma(A_a)$ is treated similarly.

Now we suppose that $\lambda\in\sigma(A_a)\cap \sigma(A_b)\cap \sigma(B)$. Then the operator $\A$ has two  linearly independent eigenvectors $U=(u_\lambda,0,0)$ and $V=(0,0, v_\lambda)$. Note also that $\A$ has no eigenvectors $Y=(u,w,v)$, where $w$ is different from zero. In this case, values  $w(-1)$ and $w(1)$ are always different from zero and hence the problems for $u$ and $v$ are unsolvable.
We will prove that $X_\lambda=\ker(\A-\lambda)^2$ and $\dim X_\lambda=3$. Let us consider the equation $(\A-\lambda)Y=c_1 U+c_2 V$ with arbitrary constants $c_1$ and $c_2$,
that is to say,
 \begin{align}\label{UstarProblemM3}
    &-u''+qu-\lambda r u=c_1 r u_\lambda \quad \text{in } \Ia, \qquad \ell_au=0,\quad u(0)=w(-1);\\\label{WstarProblemM3}
&-w''-\lambda  h w=0\quad \text{in }\cJ, \quad\qquad w'(-1)=0,\quad w'(1)=0;\\\label{VstarProblemM3}
    &-v''+qv-\lambda r v=c_2 r v_\lambda \quad \text{in } \Ib, \qquad v(0)=w(1),\quad \ell_bv=0.
\end{align}
Reasoning as above,  we establish that $w=c_0w_\lambda$ and problems \eqref{UstarProblemM3} and \eqref{VstarProblemM3} admit solutions simultaneously if and only if the following equalities
\begin{equation*}
    c_1=c_0 w_\lambda(-1)u'_\lambda(0), \qquad c_2=-c_0 w_\lambda(1)v'_\lambda(0)
\end{equation*}
 hold. Then the conditions $c_0\neq 0$ and
 \begin{equation*}
   c_1=
   -\frac{w_\lambda(-1)u'_\lambda(0)}{w_\lambda(1)v'_\lambda(0)}\,c_2
 \end{equation*}
ensure the existence of a root vector $Y_*$ of $\A$. Furthermore there are no other root vectors, by reasoning similar to that in the previous case.
Hence the subspace $X_\lambda$ for  a triple eigenvalue $\lambda$ is generated by the eigenvectors $U$, $V$ and the root vector $Y_*$.
\end{proof}


\section{Convergence of Spectra}

Let us denote by $\lambda^\eps_1<\lambda^\eps_2<\cdots<\lambda^\eps_n<\cdots$
the eigenvalues of problem \eqref{PertPrbEq}--\eqref{PertPrbCondB}, i.e.,  the eigenvalues of  $\Ae$. Note that each eigenvalue $\lambda^\eps_n$ is simple. Let $\lambda_1\leq\lambda_2\leq\cdots\leq\lambda_n\leq\cdots$ be
the eigenvalues of limit problem \eqref{LimitPrbYa}--\eqref{LimitPrbY=W} (or also the operator $\A$), counted with algebraic multiplicities.

\begin{thm}
For each $n\in \mathbb{N}$,  the eigenvalue $\lambda_n^\eps$ of  problem \eqref{LimitPrbYa}--\eqref{LimitPrbY=W} converges as \hbox{$\eps\to 0$} to the eigenvalue $\lambda_n$ of \eqref{LimitPrbYa}--\eqref{LimitPrbY=W} with the same number.
That is, if $\lambda$ is an eigenvalue of \eqref{LimitPrbYa}--\eqref{LimitPrbY=W} with algebraic multiplicity $m$, then there exists a neighbourhood of $\lambda$ which contains exactly $m$ eigenvalues  of \eqref{PertPrbEq}--\eqref{PertPrbCondB} for $\eps$ small enough.
\end{thm}
\begin{proof}
The theorem follows from the norm resolvent convergence of $\Ae$ proved in Theorem~\ref{ThmResolventConvergence} and some
ge\-ne\-ral results on the approximation of eigenvalues of compact operators.
Let $K$ be a compact operator in a separable  Hilbert space $\cH$.
Suppose that $\{K_\eps\}_{\eps>0}$ is a sequence of compact operators in $\cH$ such that $K_\eps\to K$ as $\eps\to 0$ in the uniform norm.
Let $\mu_1, \mu_2, \dots$ be the nonzero eigenvalues of $K$ ordered by decreasing magnitude taking account of algebraic multiplicities.
Then for each $\eps>0$ there is an ordering of the eigenvalues
$\mu_1(\eps), \mu_2(\eps), \dots$ of $K_\eps$ such that $\lim_{\eps\to 0}\mu_n(\eps)=\mu_n$, for each natural number $n$.
Suppose that $\mu$ is a nonzero eigenvalue of $K$ with algebraic multiplicity $m$ and  $\Gamma_\mu$ is a circle centered at $\mu$ which lies in $\rho(K)$ and contains no other points of $\sigma(K)$. Then, there is an $\eps_0$ such that, for $0<\eps\leq\eps_0$, there are exactly $m$ eigenvalues
(counting algebraic multiplicities) of $K_\eps$ lying inside $\Gamma_\mu$ and all points of $\sigma(K_\eps)$ are bounded away from $\Gamma_\mu$ \cite[Ch.1]{GohbergKrein1978}, \cite[Ch.XI-9]{DunfordSchwartz1963}, \cite{BrambleOsborn1973}.

We apply these results to $K=\rR(\A)$ and $K_\eps=\rR(\Ae)$.
Then we have
\begin{equation*}
  \sigma_p(\rR(\A))=\left\{
  \frac{1}{\lambda_n-\zeta}, \; n\in \mathbb{N}
  \right\}, \qquad
  \sigma_p(\rR(\Ae))=\left\{
  \frac{1}{\lambda_n^\eps-\zeta}, \; n\in \mathbb{N}
  \right\};
\end{equation*}
both eigenvalue sequences are ordered by decreasing magnitude.
Since $\Ae\to \A$ in the norm resolvent sense as $\eps\to 0$, that is,
$\|\rR(\Ae)-\rR(\A)\|\to 0$ as $\eps\to 0$,
we have the ``number-by-number'' convergence of the eigenvalues
\begin{equation*}
  \frac{1}{\lambda_n^\eps-\zeta}\to \frac{1}{\lambda_n-\zeta},\qquad \text{as }\eps\to 0,
\end{equation*}
from which the desired conclusion follows.
\end{proof}

\begin{rem}
  We expect that the estimate
\begin{equation*}
  \left|\lambda_n^\eps-\lambda_n\right|\leq C_n\sqrt{\eps}
\end{equation*}
to be correct for each $n\in \mathbb{N}$ and some constants $C_n$. However, it does not follow directly from bound \eqref{ResolventEst}, because resolvents $\rR(\A)$ and $\rR(\Ae)$ are not in general normal operators.
\end{rem}

\section{Some Remarks On Eigenfunction Convergence}
Since the multiplicity of eigenvalues of the limit operator is up to $3$,
the bifurcation pictures for  multiple eigenvalues of \eqref{LimitPrbYa}--\eqref{LimitPrbY=W} are quite complicated.
The bifurcations of eigenvalues as well the eigensubspaces can be described
by a more accurate asymptotic analysis. We omit the details here, because we will consider these questions in a forthcoming publication. However we can obtain some results on the limit behaviour of eigenfunctions   that follow directly from the norm resolvent convergence $\Ae\to\A$.

Let us return to compact the operators $K$ and $K_\eps$ which  appeared in the previous section. We consider the Riesz spectral projections
\begin{equation*}
  E(\mu)=\frac{1}{2\pi i}\int_{\Gamma_\mu} \rRz(\A)\,dz, \qquad
  E_\eps(\mu)=\frac{1}{2\pi i}\int_{\Gamma_\mu} \rRz(\Ae)\,dz.
\end{equation*}
The range $R(E(\mu))$ of $E(\mu)$ is the space of generalized eigenfunctions of $K$ corresponding to $\mu$ and $R(E_\eps(\mu))$ is the direct sum of the subspaces of generalized eigenfunctions of $K_\eps$ associated with the eigenvalues of $K_\eps$ inside  $\Gamma_\mu$. If $K_\eps\to K$ as $\eps\to 0$ in the  norm, then $E_\eps(\mu)\to E(\mu)$ in the norm, and therefore $\dim R(E_\eps(\mu))=\dim R(E(\mu))=m$, where $m$ is the algebraic multiplicity of $\mu$.

\begin{thm}
Let $y_{\eps,n}$ be the   eigenfunction of  \eqref{PertPrbEq}--\eqref{PertPrbCondB} which corresponds to  the eigenvalue $\lambda_n^\eps$ and  $\|y_{\eps,n}\|_{L_2(r,\cI)}=1$.

Suppose that $\lambda_n^\eps\to \lambda_n$, where $\lambda_n$ is a simple eigenvalue of $\A$ belonging to $\sigma(A_a)$. Then the eigenfunction $y_{\eps,n}$ converges in $L_2(\cI)$ as $\eps\to 0$ to the function
\begin{equation*}
  y(x)=
  \begin{cases}
    u_n(x),& \text{if }x\in\Ia,\\
    0,& \text{if }x\in\Ib
  \end{cases},
\end{equation*}
where $u_n$  is an normalized eigenfunction of $A_a$ associated with $\lambda_n$,  that is,
\begin{equation*}
  -u_n''+qu_n=\lambda_n ru_n\quad \text{in } \Ia,\quad
 \ell_au_n=0,\quad u_n(0)=0, \qquad \|u_n\|_{L_2(r,\Ia)}=1.
\end{equation*}
Similarly if  $\lambda_n$ belongs to $\sigma(A_b)$ and $\lambda_n$ is simple, then $y_{\eps,n}\to y$ in $L_2(\cI)$ as $\eps\to 0$, where
\begin{equation*}
  y(x)=
  \begin{cases}
    0,& \text{if }x\in\Ia,\\
    v_n(x),& \text{if }x\in\Ib
  \end{cases}
\end{equation*}
and $v_n$  is an normalized eigenfunction of $A_b$ with eigenvalue $\lambda_n$, i.e.,
\begin{equation*}
-v_n''+qv_n=\lambda_n rv_n\quad \text{in } \Ib,\quad
v_n(0)=0,\quad \ell_bv_n=0,\qquad \|v_n\|_{L_2(r,\Ib)}=1.
\end{equation*}

Assume $\lambda_n^\eps\to \lambda_n$, where $\lambda_n$ is a simple eigenvalue of $\A$ belonging to  $\sigma(B)$. Then the eigenfunction $y_{\eps,n}$ converges in $L_2(\cI)$ to  a solution $y$ of the problem
\begin{gather*}
  -y''+qy=\lambda_n ry\quad \text{in } \cI\setminus\{0\},\quad\ell_ay=0,\quad
 \ell_by=0,\\
y(-0)=\theta w_n(-1), \quad y(+0)= \theta w_n(1),
\end{gather*}
where  $w_n$ is the corresponding eigenfunction of $B$ such that $\|w_n\|_{L_2( h,\cJ)}=1$. Normalizing factor $\theta$ is given by
\begin{equation*}
  \theta=\left(\|T_a(\lambda_n)w_n\|^2_{L_2(r,\Ia)}
  +\|T_b(\lambda_n)w_n\|^2_{L_2(r,\Ib)}\right)^{-1}.
\end{equation*}

\end{thm}
\begin{proof}
In the case when $K=\rR(\A)$,  $K_\eps=\rR(\Ae)$, $\lambda$ is a unique point of $\sigma(\A)$ lying inside $\Gamma_\lambda$, and $\eps$ is small enough, we see  that $X_\lambda=R(E(\frac{1}{\lambda-\zeta}))$ is a subspace of generalized eigenfunctions of $\A$ corresponding to the eigenvalue $\lambda$, and  subspace $X_\lambda^\eps=R(E_\eps(\frac{1}{\lambda-\zeta}))$  is  generated by all eigenfunctions of $\Ae$ for which $\lambda^\eps_n\to \lambda$ as $\eps\to 0$.
Then the norm resolvent convergence $\Ae\to\A$ implies that  the gap between $X_\lambda^\eps$ and $X_\lambda$ tends to zero as $\eps\to 0$ for each $\lambda\in \sigma_p(\A)$. In particular, if $\lambda_n$ is a simple eigenvalue of $\A$ with eigenvector $Y_n$ and $Y_{\eps,n}$ is an eigenvector of $\Ae$ that corresponds to $\lambda_n^\eps$, then
$Y_{\eps,n}\to Y_n$ in $\cL$ as  $\eps\to 0$, provided $\|Y_{\eps,n}\|_{\cL}=\|Y_n\|_{\cL}=1$.

Assume $\lambda_n$ is a simple eigenvalue of $\A$ and $\lambda_n\in \sigma(A_a)$. In view of Theorem~\ref{ThmSpA}, subspace  $X_\lambda$ is generated by vector $Y_n=(u_n,0,0)$.
Then $Y_{\eps,n}\to Y_n$ as $\eps\to 0$ in the norm of $\cL$. If we set $Y_{\eps,n}=(y^a_\eps, w_\eps, y^b_\eps)$, then the eigenfunction $y_{\eps,n}$ of \eqref{PertPrbEq}--\eqref{PertPrbCondB} can be written as
\begin{equation*}
  y_{\eps,n}(x)=
  \begin{cases}
    y^a_\eps(x),& \text{if }x\in\Iae,\\
    w_\eps(x/\eps), & \text{if }x\in(-\eps,\eps),\\
    y^b_\eps(x), & \text{if }x\in\Ibe.
  \end{cases}
\end{equation*}
So we have
\begin{multline*}
  \|y_{\eps,n}-y_n\|^2_{L_2(\cI)}=\int_a^{-\eps}|y^a_\eps-u_n|^2\,dx
  +\int^b_{\eps}|y^a_\eps|^2\,dx
  \\
  + \int_{-\eps}^0|w_\eps(\tfrac{x}{\eps})-u_n(x)|^2\,dx
  +\int^{\eps}_0\kern-2pt|w_\eps(\tfrac{x}{\eps})|^2\,dx
  \leq c_1 \|y^a_\eps-u_n\|^2_{L_2(r,\Ia)}
  \\
  +c_2\|y^b_{\eps}\|^2_{L_2(r,\Ib)}
  + c_3\eps\|w_\eps\|^2_{L_2(h,\cJ)}+\int_{-\eps}^0 |u_n|^2\,dx
  \leq c_4\|Y_{\eps,n}-Y_n\|^2_{\cL}+c_5\eps.
\end{multline*}
The right-hand side tends to zero as $\eps\to 0$, since $Y_{\eps,n}\to Y_n$ in $\cL$ and $u_n$ is bounded on $\Ia$ as an element of $W_2^2(\Ia)$.
The same proof works for the cases $\lambda_n\in \sigma(A_b)$ and $\lambda_n\in \sigma(B)$.
\end{proof}

\begin{rem}
Of course, in the case of multiple eigenvalues, we also have some information about the convergence of eigenfunctions. For instance, if we suppose that $\lambda\in\sigma(A_a)\cap \sigma(A_b)$, but $\lambda$ is not an eigenvalue of $B$, and two eigenvalues $\lambda_n^\eps$ and $\lambda_{n+1}^\eps$ tend to $\lambda$ as $\eps\to 0$, then the gap between the eigensubspace $X_\lambda$ of $\A$ and the subspace $X_\lambda^\eps=\spn\{y_{\eps,n}, y_{\eps,n+1}\}$
vanishes as $\eps\to 0$. Therefore   eigenfunctions $y_{\eps,n}$ and $y_{\eps,n+1}$ converge in $L_2(\cI)$ to some linear combinations  $c_1u_\lambda+c_2v_\lambda$, where $u_\lambda$ and $v_\lambda$ are eigenfunctions of $A_a$ and $A_b$ respectively that correspond to $\lambda$. However, without a deeper analysis of the problem, we will not know what the linear combinations are limit positions of vectors $y_{\eps,n}$ and $y_{\eps,n+1}$ in plane $X_\lambda$.
\end{rem}

\end{document}